\documentclass[draft,reqno,11pt]{amsart}

  \usepackage[hmargin=2cm, vmargin=4cm]{geometry}

\usepackage{amsfonts}
\usepackage{amssymb}
\usepackage{amsmath}
\usepackage{amsthm}
\usepackage{latexsym}
\usepackage{graphicx}
\usepackage{layout}
\usepackage{color}
\usepackage{geometry}

\newtheorem{theorem}{Theorem}[section]

\newtheorem{proposition}[theorem]{Proposition}
\newtheorem{lemma}[theorem]{Lemma}

\renewcommand{\>}{\right\rangle}
 \renewcommand{\(}{\left(}
\renewcommand{\)}{\right)}
\renewcommand{\[}{\left[}
\renewcommand{\]}{\right]}
\newcommand{\eps}{\varepsilon}
\newcommand{\de}{\delta}

 \newcommand{\rr}{ \mathbb{R}}
 \newcommand{\di}{ \mathbf{d}}
 \newcommand{\ti}{ \mathbf{t}}

\begin{document}

\title[Bubble concentration on spheres for supercritical elliptic problems]{Bubble concentration on spheres for supercritical elliptic problems}
\author{Filomena Pacella}
\address[Filomena Pacella] {Dipartimento di Matematica "G.Castelnuovo", Universit\`{a} di Roma ``La Sapienza", P.le Aldo Moro 5, 00185 Roma, Italy}
\email{pacella@mat.uniroma1.it}

\author{Angela Pistoia}
\address[Angela Pistoia] {Dipartimento SBAI, Universit\`{a} di Roma ``La Sapienza", via Antonio Scarpa 16, 00161 Roma, Italy}
\email{pistoia@sbai.uniroma1.it}

\begin{abstract}
We consider the supercritical Lane-Emden problem
$$(P_\eps)\qquad
 -\Delta v= |v|^{p_\eps-1} v \ \hbox{in}\ \mathcal{A} ,\quad u=0\ \hbox{on}\ \partial\mathcal{A} $$
 where $\mathcal A$ is an annulus in $\rr^{2m},$ $m\ge2$ and $p_\eps={(m+1)+2\over(m+1)-2}-\eps$, $\eps>0.$

 We prove the existence of positive and sign changing solutions of $(P_\eps)$ concentrating and blowing-up,
 as $\eps\to0$, on $(m-1)-$dimensional spheres. Using a reduction method (\cite{RS,PS})
 we transform problem $(P_\eps)$ into a nonhomogeneous problem in an annulus $\mathcal D\subset \rr^{m+1}$ which can be solved by a Ljapunov-Schmidt finite
 dimensional reduction.

 \end{abstract}


\date{\today}

 \maketitle

\maketitle
\section{Introduction}

In this paper we address the question of finding solutions concentrated on manifolds of positive dimension of supercritical elliptic problems
of the type

  \begin{equation}
 \label{prob1}
 -\Delta v= |v|^{p-1} v \ \hbox{in}\ \mathcal{A} ,\quad u=0\ \hbox{on}\ \partial\mathcal{A} ,
 \end{equation}
 where $\mathcal{A} :=\{y\in \rr^{d}\ :\ a< |y|< b\} ,$ $a>0,$   is an annulus   in $\rr^{d} ,$ $d>2$
 and $p>{d+2\over d-2}$ is a supercritical exponent.

 We remark that the critical and supercritical Lane-Emden problems are very delicate due to topological and geometrical obstruction
 enlightened by the Pohozaev's identity (\cite{po}). We also point out that in the supercritical case a result of Bahri-Coron type (\cite{BC}) cannot hold in general
 nontrivially topological domains as shown by a nonexistence result of Passaseo (\cite{pa}), obtained exploiting critical exponents in lower dimensions.
 Using similar ideas, some results for exponents $p$ which are subcritical in dimension $n<d$ and instead supercritical in dimension $d$ have been obtained in different kind of domains in \cite{ACP,BCGP,CFP,DMP,GG,KP1,KP2,PPS}.

 Here we consider annuli in even dimension $d=2m,$ $m\ge2$ and obtain results about the existence of solutions, both positive and sign changing, of
 different type, concentrated on $(m-1)-$dimensional spheres. More precisely, we have

\begin{theorem}
\label{1.1} [Case of positive solutions] Let $\mathcal A\subset \rr^{2m},$ $m\ge2$ and define $(\partial \mathcal A)_a:=\left\{y\in\partial\mathcal A\ :\ |y|=a\right\}.$
There exists $\epsilon _{0}>0$ such that  for any $\epsilon
\in (0,\epsilon _{0}),$ the following supercritical problem
 \begin{equation}
 \label{prob2}
-\Delta v= |v|^{p_\eps-1} v \ \hbox{in}\ \mathcal{A} ,\quad u=0\ \hbox{on}\ \partial\mathcal{A} ,
 \end{equation}
 with $p_\eps={(m+1)+2\over(m+1)-2}-\eps$ has:
 \begin{itemize}
\item [{i)}] a positive solution $v_\eps$ which concentrates and blows-up on a $(m-1)-$dimensional sphere $\Gamma\subset(\partial \mathcal A)_a$ as $\eps\to0,$
\item [{ii)}]  a positive solution $v_\eps$ which concentrates and blows-up on two $(m-1)-$dimensional orthogonal spheres $\Gamma_1\subset(\partial \mathcal A)_a$
and $\Gamma_2\subset(\partial \mathcal A)_a$ as $\eps\to0,$
 \end{itemize}
\end{theorem}

\begin{theorem}
\label{1.2} [Case of sign changing solutions] Let $\mathcal A\subset \rr^{2m},$ $m\ge2$ and define $(\partial \mathcal A)_a:=\left\{y\in\partial\mathcal A\ :\ |y|=a\right\}.$
There exists $\epsilon _{0}>0$ such that  for any $\epsilon
\in (0,\epsilon _{0}),$ the   supercritical problem \eqref{prob2} with
   $p_\eps={(m+1)+2\over(m+1)-2}-\eps$ has:
 \begin{itemize}
\item [{i)}] a  sign changing solution $v_\eps$ such that $v_\eps^+$ and $ v_\eps^-$   concentrate  and blow-up on   two $(m-1)-$dimensional orthogonal spheres $\Gamma_+\subset(\partial \mathcal A)_a$
and $\Gamma_-\subset(\partial \mathcal A)_a,$ respectively, as $\eps\to0,$
\item [{ii)}]  a  sign changing solution $v_\eps$ such that $v_\eps^+$ and $ v_\eps^-$   concentrate  and blow-up on   the same $(m-1)-$dimensional  sphere  $\Gamma \subset(\partial \mathcal A)_a,$   as $\eps\to0,$
  \item [{iii)}] two  sign changing solutions $v^1_\eps$ and $v^2_\eps$ each one is such that $(v_\eps^i)^+$ and $ (v_\eps^i)^-$   concentrate  and blow-up on   two $(m-1)-$dimensional orthogonal spheres $(\Gamma_i)_+\subset(\partial \mathcal A)_a$
and $(\Gamma_i)_-\subset(\partial \mathcal A)_a,$ respectively, as $\eps\to0,$ $i=1,2.$
 \end{itemize}
\end{theorem}

 We remark that the exponent ${(m+1)+2\over(m+1)-2}-\eps$ which is almost critical in dimension $(m+1)$ is obviously supercritical for problem \eqref{prob2}.

 \medskip
 To prove our results we use the reduction method introduced in
  \cite{PS} which allows to transform symmetric solutions to \eqref{prob2} into symmetric solutions of a similar nonhomogeneous problem in an annulus $\mathcal D\subset \rr^{m+1}.$ This method was inspired by the paper \cite{RS} where a reduction approach was used to pass from a singularly perturbed problem in an annulus in $\rr^4$ to a singularly perturbed problem in an annulus in $\rr^3.$

  More precisely let us consider the annulus  $\mathcal D\subset \rr^{m+1} $ $\mathcal{D} :=\{x\in \rr^{m+1}\ :\ {a^2/2}< |y|< {b^2/2}\} ,$
  and, write a point $y\in \rr^{2m}$ as
    $y=(y_1,y_2)$, $y_i\in\rr^m,$ $i=1,2.$ Then we consider functions $v$ in $\mathcal A\subset \rr^{2m}$ which are radially symmetric in $y_1$ and $y_2$, i.e. $v(y)=w(|y_1|,|y_2|)$ and functions $u$ in $\mathcal D\subset \rr^{m+1}$ which are radially symmetric about the $x_{m+1}-$axis, i.e. $u(x)=h(|x|,\varphi)$
    with $\varphi=\arccos \({x\over|x|,\underline e_{m+1}}\)$ where $\underline e_{m+1}=(0,\dots,0,1).$ We also set
    $$X=\left\{v\in C^{2,\alpha}(\overline A)\ :\ \hbox{$v$ is radially symmetric}\right\}$$
  $$Y=\left\{u\in C^{2,\alpha}(\overline D)\ :\ \hbox{$u$ is axially symmetric}\right\}.$$
  Then, as corollary of Theorem 1.1 of \cite{PS} we have
  \begin{proposition}\label{prop1.1}
  There is a bijective correspondence $h$ between solutions $v$ of \eqref{prob2} in $X$ and solutions $u=h(v)$ in $Y$ of the following reduced problem
  \begin{equation}
 \label{prob3}
 -\Delta u={1\over  2|x|}|u|^{p_\eps-1 } u \quad \hbox{in}\ \mathcal D\subset \rr^{m+1},\qquad u=0\quad \hbox{on}\ \partial\mathcal D.
 \end{equation}
    \end{proposition}
As a consequence of this result we can obtain solutions of problem \eqref{prob2} by constructing axially symmetric solutions of the lower-dimensional
problem \eqref{prob3}. This has the immediate advantage of transforming the supercritical problem \eqref{prob2} into the subcritical problem \eqref{prob3} if the exponent
$p_\eps$ is taken as  $ {(m+1)+2\over(m+1)-2}-\eps.$   Indeed we will prove Theorem \ref{1.1} and Theorem \ref{1.2} by constructing axially symmetric solutions of \eqref{pro3}, positive or sign changing, which blow-up and concentrate in points of the annulus $\mathcal D\subset \rr^{m+1}.$ These solutions will give rise to solutions of \eqref{prob2} concentrating on $(m-1)-$dimensional spheres, because, as a consequence of the proof of Theorem 1.1 of \cite{PS} and of Remark 3.1 of \cite{PS} it holds
\begin{proposition}\label{prop1.2}
If $u_\eps$ is an axially symmetric solution of \eqref{prob2} concentrating, as $\eps\to0$, on a point $\xi$ which belongs to the $x_{(m+1)}-$axis, i.e. $\xi=(0,\dots,0,t)$
for $t\in\rr\setminus\{0\},$ then the corresponding solution $v_\eps=h^{-1}(u_\eps)$ concentrates on a  $(m-1)-$dimensional sphere in $\rr^{2m}.$
\end{proposition}

This is because, by symmetry considerations and by the change of variable performed in \cite{PS} to prove Theorem 1.1 any point $\xi$ on the $x_{(m+1)}-$axis in $\mathcal D\subset\rr^{m+1}$ is mapped into a $(m-1)-$dimensional sphere in $\mathcal A\subset\rr^{2m}.$ We refer to \cite{PS} for all details.

Thus let $\Omega:=\{x\in \rr^n\ :\ 1< |x|< r\} $   be an annulus   in $\rr^n,$  $n \ge3,$ and consider the problem
   \begin{equation}
 \label{p1}
 -\Delta u={1\over  2|x|}|u|^{p-1-\epsilon} u \quad \hbox{in}\ \Omega,\qquad u=0\quad \hbox{on}\ \partial\Omega,
 \end{equation}
 where   $p={n+2\over n-2}  $ and $\eps$ is a small positive parameter.
 Let
 $U_{\delta,\xi}(x):=\alpha_n{\delta^{n-2\over2}\over(\delta^2+|x-\xi|^2){n-2\over2}}$ with $ \delta>0$ and $ x,\xi\in\rr^n,$
 be the solutions to the critical problem $-\Delta u=u^p$ in $\rr^n.$ Here $ \alpha_n:=[n(n-2)]^{n-2\over4}.$
We have
\begin{theorem}
\label{main} There exists $\epsilon _{0}>0$ such that, for each $\epsilon
\in (0,\epsilon _{0}),$ problem \eqref{p1} has
\begin{itemize}
\item[(i)] an axially symmetric positive
solution $u_{\epsilon }$ with one simple positive blow-up point which converge to $\xi_0$ as $\eps$ goes to zero, i.e.
\begin{equation*}
u_{\epsilon }(x)=U_{{\delta  }_{\epsilon },{\xi  }_{\epsilon }}(x)  +o(1)\ \quad \hbox{in}\ H^1_0(\Omega),
\end{equation*}%
with
\begin{equation*}
\epsilon ^{-{\frac{n-1}{n-2}}}{\delta  }_{\epsilon }\rightarrow
d >0,\quad  {\xi  }_{\epsilon }\rightarrow \xi_0 ;
\end{equation*}%
\item[(ii)] an axially symmetric positive
solution $u_{\epsilon }$ with two simple positive blow-up points which converge to $\xi_0$ and $-\xi_0$ as $\eps$ goes to zero, i.e.
\begin{equation*}
u_{\epsilon }(x)=U_{ \delta   _ \epsilon  , \xi  _ \epsilon  }(x)+U_{ \delta _ \epsilon  , -\xi _ \epsilon  }(x)+o(1),
\end{equation*}%
with
\begin{equation*}
\epsilon ^{-{\frac{n-1}{n-2}}}{\delta  }_{\epsilon }\rightarrow
d >0,\quad  {\xi  }_{\epsilon }\rightarrow \xi_0 ;
\end{equation*}%
\item[(iii)] an axially symmetric sign-changing solutions
solution $u_{\epsilon }$ with one  simple positive and one simple negative blow-up points which converge to $\xi_0$ and $-\xi_0$ as $\eps$ goes to zero, i.e.
\begin{equation*}
u_{\epsilon }(x)=U_{ \delta   _ \epsilon  , \xi  _ \epsilon  }(x)-U_{ \delta _ \epsilon  , -\xi _ \epsilon  }(x)+o(1),
\end{equation*}%
with
\begin{equation*}
\epsilon ^{-{\frac{n-1}{n-2}}}{\delta  }_{\epsilon }\rightarrow
d >0,\quad  {\xi  }_{\epsilon }\rightarrow \xi_0 ;
\end{equation*}%
\item[(iv)] an axially symmetric sign-changing solutions
solution $u_{\epsilon }$ with one  double nodal blow-up point  which converge to $\xi_0$  as $\eps$ goes to zero, i.e.
\begin{equation*}
u_{\epsilon }(x)=U_{{\delta _{1}}_{\epsilon },{\xi _{1}}_{\epsilon }}(x)-U_{{%
\delta _{2}}_{\epsilon },{\xi _{2}}_{\epsilon }}(x)+o(1),
\end{equation*}%
with
\begin{equation*}
\epsilon ^{-{\frac{n-1}{n-2}}}{\delta _{i}}_{\epsilon }\rightarrow
d_{i}>0,\quad  {\xi _{i}}_{\epsilon }\rightarrow \xi_0
\end{equation*}%
for $i=1,2.$
\item[(v)] two  axially symmetric sign-changing solutions
solution $u_{\epsilon }$ with two  double nodal blow-up points  which converge to $\xi_0$ and $-\xi_0$ as $\eps$ goes to zero, i.e.
\begin{equation*}
u_{\epsilon }(x)=\[U_{{\delta _{1}}_{\epsilon },{\xi _{1}}_{\epsilon }}(x)-U_{{%
\delta _{2}}_{\epsilon },{\xi _{2}}_{\epsilon }}(x)\]+\[U_{{-\delta _{1}}_{\epsilon },{-\xi _{1}}_{\epsilon }}(x)-U_{{%
-\delta _{2}}_{\epsilon },{-\xi _{2}}_{\epsilon }}(x)\]+o(1)
\end{equation*}%
and
\begin{equation*}
u_{\epsilon }(x)=\[U_{{\delta _{1}}_{\epsilon },{\xi _{1}}_{\epsilon }}(x)-U_{{%
\delta _{2}}_{\epsilon },{\xi _{2}}_{\epsilon }}(x)\]-\[U_{{-\delta _{1}}_{\epsilon },{-\xi _{1}}_{\epsilon }}(x)-U_{{%
-\delta _{2}}_{\epsilon },{-\xi _{2}}_{\epsilon }}(x)\]+o(1)
\end{equation*}
with
\begin{equation*}
\epsilon ^{-{\frac{n-1}{n-2}}}{\delta _{i}}_{\epsilon }\rightarrow
d_{i}>0,\quad  {\xi _{i}}_{\epsilon }\rightarrow \xi_0
\end{equation*}%
for $i=1,2.$
\end{itemize}
\end{theorem}

Obviously Theorem \ref{1.1} and Theorem \ref{1.2} derive from Theorem \ref{main} for $n=m+1$ using Proposition \ref{prop1.1} and Proposition \ref{prop1.2}.

The proof of Theorem
\ref{main} relies on a very well known Ljapunov-Schmidt
finite dimensional reduction. We omit many details on the finite dimensional reduction because
they can be found, up to some minor modifications, in the literature. In Section \ref{uno} we
write the approximate solution, we sketch the proof of the Ljapunov-Schmidt
procedure and we prove Theorem  \ref{main}. In Section \ref{due} we only
compute the expansion of the reduced energy, which is crucial in this framework. In the Appendix  we recall
some well known facts.

 \section{The Ljapunov-Schmidt procedure}\label{uno}
 We equip ${\rm H}^1_0(\Omega)$ with the inner product $(u,v)=\int\limits_\Omega \nabla u\nabla vdx$ and the corresponding norm $\|u\|^2= \int\limits_\Omega |\nabla u|^2dx .$ For $r\in[1,\infty)$ and $u\in{\rm L}^{r}(\Omega)$ we set $\|u\|_r^r= \int\limits_\Omega | u|^rdx .$

 Let us rewrite problem \eqref{p1} in a different way. Let $i^*:{\rm L}^{2n\over n-2}(\Omega)\to{\rm H}^1_0(\Omega)$ be the adjoint operator of the embedding
 $i:{\rm H}^1_0(\Omega) \hookrightarrow{\rm L}^{2n\over n-2}(\Omega),$ i.e.
 $$i^*(u)=v\quad \Leftrightarrow\quad \(v,\varphi\)=\int\limits_\Omega u(x)\varphi(x)dx\ \forall\ \varphi\in{\rm H}^1_0(\Omega).$$
  It is clear that there exists a positive constant $c$ such that
 $$\|i^*(u)\|\le c\|u\|^{2n\over n+2}\qquad  \forall\ u\in{\rm L}^{2n\over n+2}(\Omega).$$
 Setting $f_\eps(s):=|s|^{p-1-\eps}s$ and using the operator $i^*$, problem \eqref{p1} turns out to be equivalent to
 \begin{equation}
 \label{p2}
 u=i^*\[{1\over 2|x|}f_\eps(u)\],\quad u\in {\rm H}^1_0(\Omega).
 \end{equation}

 Let $U_{\delta,\xi}:=\alpha_n{\delta^{n-2\over2}\over(\delta^2+|x-\xi|^2)^{n-2\over2}},$ with $\alpha_n:=\[n(n-2)\]^{n-2\over4} $
 be the positive solutions to the limit problem
 $$ -\Delta u=u^p\ \hbox{in}\ \rr^n.$$
Set
$$\psi^0_{\delta,\xi}(x):={\partial U_{\delta ,\xi }\over \partial\delta }=\alpha_n{n-2\over2}\delta^{n-4\over2}{|x-\xi|^2-\delta^2\over(\delta^2+|x-\xi|^2)^{n/2}}$$
 and for any   $j= 1,\dots,n$
 $$\psi^j_{\delta,\xi}(x):={\partial U_{\delta ,\xi }\over \partial\xi_j }=\alpha_n(n-2)\delta^{n-2\over2}{ x_j-\xi_j\over(\delta^2+|x-\xi|^2)^{n/2}}.$$
It is well known that the space spanned by the $(n+1)$ functions $\psi^j_{\delta,\xi}$ is the set of the solution to the linearized problem
$$ -\Delta \psi=pU^{p-1}_{\delta,\xi}\psi\ \hbox{in}\ \rr^n.$$

 We also denote by $PW$ the projection onto ${\rm H}^1_0(\Omega)$ of a function $W\in D^{1,2}(\rr^n),$ i.e.
 $$\Delta PW=\Delta W\ \hbox{in}\ \Omega,\quad PW=0\ \hbox{on}\ \partial\Omega.$$

Set $\xi_0:=(0,\dots,0,1).$
  We look for two different types of solutions to problem \eqref{p2}.
  The solutions of the type   (i), (ii) and (iii) of Theorem \ref{main} will be of the form
  \begin{equation}
 \label{ans1}
 u_\eps= PU_{\delta ,\xi }+\lambda PU_{\mu ,\eta }+\phi
 \end{equation}
 where $\lambda \in\{-1,0,+1\}$ ($\lambda=0$ in case (i), $\lambda=+1$ in case (ii) and $\lambda=-1$ in case (iii)) and  the concentration parameters   are
 \begin{equation}
 \label{ans2}
 \mu =\delta \quad\hbox{and}\quad \delta :=\eps^{n-1\over n-2}d  \ \hbox{for some}\  d  >0
 \end{equation}
while the concentration points   satisfy
 \begin{equation}
 \label{ans3}
   \eta =-\xi \quad\hbox{and}\quad \xi =(1+\tau )\xi_0,   \ \hbox{with}\ \tau := \eps t ,\ t  >0.
 \end{equation}
 On the other hand, the solutions of the type (iv) and (v) of Theorem \ref{main} will be of the form
 \begin{equation}
 \label{ans11}
 u_\eps=PU_{\delta_1,\xi_1}-PU_{\delta_2,\xi_2}+\lambda\(PU_{\mu_1,\eta_1}-PU_{\mu_2,\eta_2}\)+\phi,
 \end{equation}
where $\lambda \in\{-1,0,+1\}$ ($\lambda=0$ in case (iv), $\lambda=+1$ in the first  case (v) and $\lambda=-1$ in the second case (v)) and  the concentration parameters are
  \begin{equation}\label{ans21}
 \mu_i=\delta_i\quad\hbox{and}\quad \delta_i:=\eps^{n-1\over n-2}d_i \quad\hbox{with}\quad d_i >0
 \end{equation}
while the concentration points are aligned along the $x_n-$axes   and satisfy
 \begin{equation}
 \label{ans31}
  \eta_i=-\xi_i\quad\hbox{and}\quad  \xi_i=(1+\tau_i)\xi_0 \ \hbox{with} \   \tau_i:= \eps t_i,\ t_i>0.
 \end{equation}

 Next, we introduce the configuration space $\Lambda$ where the concentration parameters and the concentration points lie. For solutions of type \eqref{ans1} we
 set $\di=d\in(0,+\infty)$ and $\ti=t\in(0,+\infty)$
 and so
 $$\Lambda:=\left\{(\di,\ti)\in (0,+\infty)\times(0,+\infty)\right\},$$
  while for solutions of type \eqref{ans11} we
 set $\di=(d_1,d_2)\in(0,+\infty)^2$ and $\ti=(t_1,t_2)\in(0,+\infty)^2$
 and so
 $$\Lambda:=\{(\di,\ti)\in(0,+\infty)^4\ :\ t_1<t_2\} .$$
 In each of these cases we write
 $$V_{\di,\ti}:=PU_{\delta ,\xi }+\lambda PU_{\mu ,\eta }\quad\hbox{or}\quad V_{\di,\ti}:= PU_{\delta_1,\xi_1}-PU_{\delta_2,\xi_2}+\lambda\(PU_{\mu_1,\eta_1}-PU_{\mu_2,\eta_2}\).$$

 The remainder term $\phi$ in both cases \eqref{ans1} and \eqref{ans11} belongs to a suitable   space   which we now define.

  We  introduce the spaces
 $$K_{\di,\ti}:={\rm span}\{P\psi^j_{\delta_i,\xi_i},\ P\psi^\ell_{\mu_\kappa,\xi_\kappa} \ :\ i,\kappa=1,2,\ j,\ell=0,1,\dots,n\} $$
(we agree that if $\lambda=0$ then  $K_{\di,\ti}$ is only generated by the $P\psi^j_{\delta_i,\xi_i}$'s) and
 $$K_{\di,\ti}^\perp:=\left\{\phi\in \mathcal H_\lambda\ :\ (\phi, \psi )=0\quad \forall\ \psi\in K_{\di,\ti}\right\},$$
 where  the space $\mathcal H_\lambda$ depends on $\lambda \in\{-1,0,+1\}$ and is defined by
  $$\mathcal H_0:=\{\phi\in  {\rm H}^1_0(\Omega)\ :\ \phi\ \hbox{is axially symmetric with respect to the $x_n$-axes }\},$$
  $$\mathcal H_{+1}:=\{\phi\in  \mathcal H_0\ :\ \phi(x_1,\dots,x_n)=\phi(x_1,\dots,-x_n\},$$
 $$\mathcal H_{-1}:=\{\phi\in  \mathcal H_0\ :\ \phi(x_1,\dots,x_n)=-\phi(x_1,\dots,-x_n\}.$$
  Then we introduce   the orthogonal projection operators
  $\Pi_{\di,\ti}$ and $ \Pi_{\di,\ti}^\perp $ in $H^1_0(\Omega),$ respectively.

As  usual for this reduction method, the approach to solve problem \eqref{p1} or \eqref{p2} will be to find a pair $(\di,\ti)$ and a function $\phi\in K_{\di,\ti}^\perp$ such that
\begin{equation}\label{es1}
\Pi_{\di,\ti}^\perp\left\{V_{\di,\ti}+\phi-i^*\[{1\over 2|x|}f_\eps\(V_{\di,\ti}+\phi\)\]\right\}=0
\end{equation}
and
\begin{equation}\label{es2}
\Pi_{\di,\ti} \left\{V_{\di,\ti}+\phi-i^*\[{1\over 2|x|}f_\eps\(V_{\di,\ti}+\phi\)\]\right\}=0
\end{equation}

First,   we shall find for any $(\di,\ti)$ and for small $\eps,$ a function $\phi\in K_{\di,\ti}^\perp$  such that \eqref{es1} holds.
To this aim we define a linear operator $L_{\di,\ti}:K_{\di,\ti}^\perp\to K_{\di,\ti}^\perp$ by
$$L_{\di,\ti}\phi:=\phi-\Pi_{\di,\ti}^\perp i^*\[ f'_0\(V_{\di,\ti}\)\phi\].$$

\begin{proposition}\label{pro1}
For any compact sets $\mathbf{C}$ in $\Lambda$ there exists $\eps_0,c>0$ such that for any $\eps\in(0,\eps_0)$ and for any $(\di,\ti)\in \mathbf{C}$
the operator $L_{\di,\ti}$ is invertible and
$$\|L_{\di,\ti}\phi\|\ge c\|\phi\|\ \quad\ \forall\ \phi\in K_{\di,\ti}^\perp.$$
\end{proposition}
\begin{proof}

We argue as in Lemma 1.7 of \cite{MP}.

\end{proof}

Now, we are in position to solve equation \eqref{es1}.

\begin{proposition}\label{pro2}
For any compact sets $\mathbf{C}$ in $\Lambda$ there exists $\eps_0, c,\sigma>0$ such that for any $\eps\in(0,\eps_0)$ and  for any $(\di,\ti)\in \mathbf{C}$ there exists a unique $\phi ^\eps_{\di,\ti} \in K_{\di,\ti}^\perp $
such that
$$\Pi_{\di,\ti}^\perp\left\{V_{\di,\ti}+\phi^\eps_{\di,\ti}-i^*\[{1\over 2|x|}f_\eps\(V_{\di,\ti}+\phi^\eps_{\di,\ti}\)\]\right\}=0.
$$
Moreover
$$
\left\|\phi^\eps_{\di,\ti}\right\|\le c\eps^{{1\over2}+\sigma}.
$$
\end{proposition}
\begin{proof}
First, we estimate the rate of the error term
$$
R_{\di,\ti}:=\Pi_{\di,\ti}^\perp\left\{V_{\di,\ti} -i^*\[{1\over  |x|}f_\eps\(V_{\di,\ti} \)\]\right\}
$$
as
$$\left\|R_{\di,\ti}\right\|_{2n\over n+2}=O\(\eps^{{1\over2}+\sigma}\)$$
for some $\sigma>0.$
We argue as in Appendix B of \cite{ACP} using estimates of Section \ref{due}.
Then we argue exactly as in Proposition 2.3 of \cite{BMP}.\end{proof}

Now, we introduce the energy functional $J_\eps: {\rm H}^1_0(\Omega)\to\rr$ defined by
$$J_\eps(u):={1\over2}\int\limits_{\Omega}|\nabla u |^2dx-{1\over p+1-\eps}\int\limits_{\Omega}{1\over 2|x|}| u |^{p+1-\eps}dx,$$
whose critical points are the solutions to problem \eqref{p1}.
Let us define the reduced energy   $\widetilde J_\eps:\Lambda\to\rr$ by
$$\widetilde J_\eps(\di,\ti)=J_\eps\(V_{\di,\ti}+\phi^\eps_{\di,\ti}\).$$
Next, we prove that the critical points of $\widetilde J_\eps$ are the solution to problem \eqref{es2}.

\begin{proposition}\label{pro3}
 The function $V_{\di,\ti}+\phi^\eps_{\di,\ti}$ is a critical point of the functional $J_\eps$ if and only if the point $(\di,\ti)$ is a critical point of the function $\widetilde J_\eps.$
\end{proposition}
\begin{proof}

We argue as in Proposition 1 of \cite{BLR}.

\end{proof}

The problem is thus reduced to the search for critical points of $\widetilde J_\eps,$ so it is necessary to compute the asymptotic expansion of $\widetilde J_\eps$.

\begin{proposition}\label{pro4}
 It holds true that
 $$ \widetilde J_\eps(\di,\ti)= c_1+ c_2\eps+c_3\eps\log\eps+\eps(1+|\lambda|)\Phi(\di,\ti)+o(\eps),
$$
 $C^0-$uniformly on compact sets of $\Lambda,$ where
 \begin{itemize}
 \item[(i)] in case \eqref{ans1}
 $$ \Phi(\di,\ti):= c_4 \({d \over 2t }\)^{n-2} +c_5t -c_6 \ln d $$
 \item[(ii)] in case \eqref{ans11}
 \begin{align*}
 \Phi(\di,\ti):=&c_4\[\({d_1\over 2t_1}\)^{n-2}+\({d_2\over 2t_2}\)^{n-2}+2\(d_1d_2\)^{n-2\over2}\({1\over |t_1-t_2|^{n-2}}-{1\over |t_1+t_2|^{n-2}}\)\]
 \nonumber\\ &+c_5\(t_1+t_2\)-c_6\(\ln d_1+\ln d_2\).\end{align*}
 \end{itemize}
Here  $c_i$'s are  constants and $c_4,$ $c_5$ and $c_6$ are positive.
\end{proposition}
\begin{proof}

The proof is postponed to Section \ref{due}.
 \end{proof}

\begin{proof}[Proof of Theorem \ref{main}]
It is easy to verify that the function $\Phi$ of Proposition \ref{pro4} in both cases
has a minimum point which is stable under uniform perturbations. Therefore, if $\eps$ is small enough there exists a critical point $(\di_\eps,\ti_\eps)$
of the reduced energy $ \widetilde J_\eps.$ Finally, the claim follows by Proposition \ref{pro3}.

\end{proof}

 \section{Expansion of the reduced energy}\label{due}

It is standard to prove that
$$\widetilde J_\eps(\di,\ti)=J_\eps\(V_{\di,\ti}\)+ o(\eps)$$
(see for example \cite{BLR,BMP}).
So the problem reduces to estimating the leading term $J_\eps\(V_{\di,\ti}\).$
We will estimate it in case \eqref{ans11} with $|\lambda|=1$, because in   the other cases  its expansion is easier
and can be deduced from that. Proposition \ref{pro4} will follow from Lemma \ref{lex1},  Lemma \ref{ley1} and  Lemma \ref{lem33}.

For future reference we define the constants
\begin{align}\label{g1}
&\gamma_1=\alpha_n^{p+1}\int\limits_{\rr^n}{1\over (1+|y|^2)^n}dy,\\
\label{g2}
&\gamma_2=\alpha_n^{p +1}\int\limits_{\rr^n}{1\over (1+|y|^2) ^{n+2\over2}}dy,\\
\label{g3}
&\gamma_3=\alpha_n^{p+1 }\int\limits_{\rr^n}{1\over (1+|y|^2) ^{n }}\log {1\over (1+|y|^2 )^{n-2\over2}}dy.
\end{align}

For sake of simplicity, we set $U_i:=U_{\de_i,\xi_i}$ and $V_i:=V_{\mu_i,\eta_i}.$

\begin{lemma}\label{lex1}
It holds true that
\begin{align*}&{1\over2}\int\limits_\Omega|\nabla V_{\di,\ti}|^2dx=2 \gamma_1 \\&- \gamma_2\eps\[\({d _1\over2t_1}\)^{n-2}+\({d_2\over2t_2}\)^{n-2}
+\(d_1d_2\)^{n-2\over2}\({1\over|t_1-t_2|^{n-2}}-{1\over|t_1+t_2|^{n-2}}\)\] +o(\eps).
\end{align*}
\end{lemma}
\begin{proof}
We have
\begin{align}\label{lex11}
\int\limits_\Omega|\nabla V_{\di,\ti}|^2dx=&\int\limits_\Omega|\nabla PU_1|^2dx+\int\limits_\Omega|\nabla PU_2|^2dx-2\int\limits_\Omega \nabla PU_1\nabla PU_2dx\nonumber\\
&+\int\limits_\Omega|\nabla PV_1|^2dx+\int\limits_\Omega|\nabla PV_2|^2dx-2\int\limits_\Omega \nabla PV_1\nabla PV_2dx\nonumber\\
&+2\sum\limits_{i,j=1}^2\lambda\int\limits_\Omega \nabla PU_i  \nabla PV_jdx\nonumber\\
=&2\(\int\limits_\Omega|\nabla PU_1|^2dx+\int\limits_\Omega|\nabla PU_2|^2dx-2\int\limits_\Omega \nabla PU_1\nabla PU_2dx\)+o(\eps),\nonumber\\
\end{align}
because of the symmetry (see \eqref{ans21} and \eqref{ans31}) and the fact that
$$\int\limits_\Omega \nabla PU_i  \nabla PV_jdx=O\(\de_i^{n-2\over2}\mu_j^{n-2\over2}\)=o(\eps).$$
Let us estimate the first term in \eqref{lex11}. The estimate of the second term   is similar.
We set
\begin{equation}\label{tau}\tau:=\min\left\{{\rm d}(\xi_1,\partial\Omega),{\rm d}(\xi_2,\partial\Omega),{|\xi_1-\xi_2|\over2}\right\}=\min\left\{\tau_1,\tau_2,{|\tau_1-\tau_2|\over2}\right\}.\end{equation}
We get
$$ \int\limits_\Omega|\nabla PU_1|^2dx=\int\limits_\Omega U_1^pPU_1dx=\int\limits_{B(\xi_1,\tau )}U_1^pPU_1dx+\int\limits_{\Omega\setminus B(\xi_1,\tau )}U_1^pPU_1dx.
$$

By Lemma \ref{lem2} we deduce
$$
 \int\limits_{\Omega\setminus B(\xi_1,\tau )} U_1^pPU_1dx=O\(\({\delta_1\over\tau  }\)^n\)=o(\eps)
$$

\begin{align}\label{lex14}
 &\int\limits_{B(\xi_1,\tau )} U_1^pPU_1dx=    \int\limits_{B(\xi_1,\tau )} U_1^{p+1}dx +\int\limits_{B(\xi_1,\tau )} U_1^p\(PU_1-U_1\)dx,
\end{align}
with
$$
  \int\limits_{B(\xi_1,\tau)} U_1^{p+1} =\gamma_1+O\(\({\delta_1\over\tau_1}\)^n\)=\gamma_1+o(\eps).
$$
 The second term in \eqref{lex14} is estimated in (i) of Lemma \ref{lez1}.

It remains only to estimate the third term in \eqref{lex11}.

\begin{align}\label{lex18}
\int\limits_\Omega \nabla PU_1\nabla PU_2dx=\int\limits_\Omega   U_1^p  PU_2dx=\int\limits_{B(\xi_1,\tau )} U_1^p  PU_2dx+
\int\limits_{\Omega\setminus B(\xi_1,\tau )} U_1^p  PU_2dx.
\end{align}

We have
\begin{align*}
 &\int\limits_{\Omega\setminus B(\xi_1,\tau )} U_1^p PU_2dx=O\({\delta_1^{n+2\over2}\delta_2^{n-2\over2} } \int\limits_{\Omega\setminus B(\xi_1,\tau )}{1\over|x-\xi_1|^{n+2}}{1\over|x-\xi_2|^{n-2}}dx\)\nonumber\\ &
=O\({\delta_1^{n+2\over2}\delta_2^{n-2\over2}\over\tau^n } \int\limits_{\rr^n\setminus B(0,1 )}{1\over|y|^{n+2}}{1\over|y+{\xi_1-\xi_2\over\tau}|^{n-2}}dy\)=O\({\delta_1^{n+2\over2}\delta_2^{n-2\over2}\over\tau^n }\)=o(\eps).
\end{align*}

 The first term in \eqref{lex18} is estimated in (ii) of Lemma \ref{lez1}.

The claim then follows collecting all the previous estimates and  taking into account the choice of $\delta_i'$s and $\tau_i'$s made in \eqref{ans1} and \eqref{ans2}.

\end{proof}

\begin{lemma}\label{ley1}
It holds true that
\begin{align*}
&{1\over p+1}\int\limits_\Omega{1\over |x|}|  V_{\di,\ti}|^{p+1}dx=2\[{2\over p+1} \gamma_1-{1\over p+1} \gamma_1\eps\(t_1+t_2\)\]\\&-2\gamma_2\eps\[\({d_1\over2t_1}\)^{n-2}+\({d_2\over2t_2}\)^{n-2}
+2\(d_1d_2\)^{n-2\over2}\({1\over|t_1-t_2|^{n-2}}-{1\over|t_1+t_2|^{n-2}}\)\] +o(\eps).
\end{align*}
\end{lemma}
\begin{proof}
We have
\begin{align}\label{ley11}
& \int\limits_\Omega{1\over |x|}|  V_{\di,\ti}|^{p+1}dx=\int\limits_\Omega{1\over |x|}|  PU_1-PU_2+\lambda\(PV_1-PV_2\)|^{p+1}dx\nonumber\\ &
 =\int\limits_\Omega{1\over |x|}\(|  PU_1-PU_2+\lambda\(PV_1-PV_2\)|^{p+1}-|   U_1 |^{p+1}-|    U_2|^{p+1}-|   V_1 |^{p+1}-|    V_2|^{p+1}\)dx\nonumber\\ &
+\int\limits_\Omega{1\over |x|}\( |   U_1 |^{p+1} +|    U_2|^{p+1}+|   V_1 |^{p+1}+|    V_2|^{p+1}\)dx \nonumber\\ &
 =\int\limits_\Omega{1\over |x|}\(|  PU_1-PU_2+\lambda\(PV_1-PV_2\)|^{p+1}-|   U_1 |^{p+1}-|    U_2|^{p+1}-|   V_1 |^{p+1}-|    V_2|^{p+1}\)dx\nonumber\\ &
+2\int\limits_\Omega{1\over |x|}\( |   U_1 |^{p+1} +|    U_2|^{p+1} \)dx,
\end{align}
because of the symmetry (see \eqref{ans21} and \eqref{ans31}).

The last two terms in \eqref{ley11} are estimated in (v) of Lemma \ref{lez1}.
Let
 $\tau$ as in \eqref{tau}.

We split the first integral as
\begin{align}\label{ley12}
&  \int\limits_\Omega{1\over |x|}\(|   PU_1-PU_2+\lambda\(PV_1-PV_2\)|^{p+1}-|   U_1 |^{p+1}-|    U_2|^{p+1}-|   V_1 |^{p+1}-|    V_2|^{p+1}\)dx\nonumber\\ &
=\int\limits_{B(\xi_1,\tau)}\dots+\int\limits_{B(\xi_2,\tau)}\dots+\int\limits_{B(-\xi_1,\tau)}\dots+\int\limits_{B(-\xi_2,\tau)}\dots\nonumber\\ &+\int\limits_{\Omega\setminus\(B(\xi_1,\tau)\cup B(\xi_2,\tau)\cup B(-\xi_1,\tau)\cup B(-\xi_2,\tau)\)}\dots
\end{align}

By Lemma \ref{lem2} we deduce

\begin{align*}
&  \int\limits_{\Omega\setminus\(B(\xi_1,\tau)\cup B(\xi_2,\tau)\cup B(-\xi_1,\tau)\cup B(-\xi_2,\tau)\)}\dots\nonumber\\ &
=O\( \int\limits_{\Omega\setminus\(B(\xi_1,\tau)\cup B(\xi_2,\tau)\cup B(-\xi_1,\tau)\cup B(-\xi_2,\tau)\)}  \(    U_1  ^{p+1}+ U_2 ^{p+1}+V_1  ^{p+1}+ V_2 ^{p+1}\)dx\)
\nonumber\\ &=O\({\delta_1^n\over\tau^n}+{\delta_2^n\over\tau^n}\)=o(\eps).
\end{align*}

We now estimate  the integral over $B(\xi_1,\tau)$ in \eqref{ley12}.
\begin{align}\label{ley14}
&   \int\limits_{  B(\xi_1,\tau) }{1\over |x|}\(|   PU_1-PU_2+\lambda\(PV_1-PV_2\)|^{p+1}-|   U_1 |^{p+1}-|    U_2|^{p+1}-|   V_1 |^{p+1}-|    V_2|^{p+1}\)dx\nonumber\\ &=
(p+1) \int\limits_{  B(\xi_1,\tau) }{1\over |x|}    U_1  ^ p\(  PU_1-U_1 - PU_2+\lambda\(PV_1-PV_2\)\) dx\nonumber\\ &+
{p(p+1)\over2} \int\limits_{  B(\xi_1,\tau) }{1\over |x|}   | U_1+\theta \rho |^{p-1} \rho^2dx
  -\int\limits_{  B(\xi_1,\tau) }{1\over |x|}\( |    U_2|^{p+1} +|   V_1 |^{p+1}-|    V_2|^{p+1}\)dx\nonumber\\ &=
(p+1) \int\limits_{  B(\xi_1,\tau) }{1\over |x|}    U_1  ^ p\(  PU_1-U_1 \)dx-(p+1) \int\limits_{  B(\xi_1,\tau) }{1\over |x|}    U_1  ^ p PU_2 dx+o(\eps),
\end{align}
where $\rho:= PU_1-U_1-PU_2+\lambda\(PV_1-PV_2\)$.
Indeed, by Lemma \ref{lem2}
one can easily deduce that
$$\int\limits_{  B(\xi_1,\tau) }{1\over |x|}    U_1  ^ p    \(PV_1-PV_2 \) dx, \int\limits_{  B(\xi_1,\tau) }{1\over |x|}  |    U_2|^{p+1}  dx,\int\limits_{  B(\xi_1,\tau) }{1\over |x|}
|   V_i |^{p+1} dx=o(\eps) $$
and also
\begin{align*}
&    {p(p+1)\over2} \int\limits_{  B(\xi_1,\tau) }{1\over |x|}   | U_1+\theta\rho|^{p-1} \rho^2dx\le c\int\limits_{  B(\xi_1,\tau) }   | U_1|^{p-1} \rho^2dx+\int\limits_{  B(\xi_1,\tau) }  |\rho|^{p+1}  dx	\\ &
\nonumber\\ &\le c\int\limits_{  B(\xi_1,\tau) }     U_1  ^{p-1} \(  PU_1-U_1 \)^2dx+c\int\limits_{  B(\xi_1,\tau) }     U_1  ^{p-1}\(PU_2\)^2dx+c
\int\limits_{  B(\xi_1,\tau) }     U_1  ^{p-1} \(PV_1-PV_2 \)^2dx
\nonumber\\ &+ c\int\limits_{  B(\xi_1,\tau) }|PU_1-U_1  |^{p+1}+c\int\limits_{  B(\xi_1,\tau) }|  U_2  |^{p+1}+c\int\limits_{  B(\xi_1,\tau) }\(|  V_1  |^{p+1}+|  V_2  |^{p+1}\)dx\nonumber\\ &
=o(\eps).
\end{align*}

The first term and the second term in \eqref{ley14} are estimated in (iii) and (iv) of Lemma \ref{lez1}, respectively.

 Therefore, the claim follows.
\end{proof}

\begin{lemma}\label{lem33}
It holds true that
\begin{align*}
& {1 \over p + 1 -\eps} \int\limits_\Omega {1\over |x|}|V_{\di,\ti}|^{p+1-\eps}={1 \over p + 1} \int\limits_\Omega {1\over |x|}|V_{\di,\ti}|^{p+1}   \\
&+  \(1+|\lambda|\)\[{\gamma_1\over (p+1)^2}-\alpha_n{\gamma_1\over (p+1) }-{\gamma_3\over (p+1) }\eps +{n-2\over2(p+1)}\(\ln\de_1+\ln\de_2\)\]+ o(\eps).
\end{align*}
\end{lemma}
\begin{proof}
We argue exactly as in Lemma 3.2 of \cite{DFM}.
\end{proof}

\begin{lemma}\label{lez1} Let $\tau$ as in \eqref{tau}.
It holds true that
\begin{itemize}
\item[(i)] $$\int\limits_{  B(\xi_1,\tau) }    U_1  ^ p\(  PU_1-U_1\)dx=-\gamma_2\({\delta_1\over2\tau_1}\)^{n-2}+o(\eps) $$
\item[(ii)] $$\int\limits_{  B(\xi_1,\tau) }     U_1  ^ pPU_2dx= \gamma_2\({\delta_1\delta_2 }\)^{n-2\over2}\({1\over|\tau_1-\tau_2|^{n-2}}-{1\over|\tau_1+\tau_2|^{n-2}}\)+o(\eps)$$
\item[(iii)] $$\int\limits_{  B(\xi_1,\tau) }{1\over |x|}    U_1  ^ p\(  PU_1-U_1\)dx=-\gamma_2\({\delta_1\over2\tau_1}\)^{n-2}+o(\eps) $$
\item[(iv)] $$\int\limits_{  B(\xi_1,\tau) }{1\over |x|}    U_1  ^ pPU_2dx=-\gamma_2\({\delta_1\over2\tau_1}\)^{n-2}+o(\eps)$$
\item[(v)] $$\int\limits_\Omega{1\over |x|}U_1 ^{p+1}dx=\gamma_1-\gamma_1\tau_1+o(\eps).$$

\end{itemize}
\end{lemma}
\begin{proof}

{\em Proof of (i)}
By Lemma \ref{lem2} we get
\begin{align*}
 &\int\limits_{B(\xi_1,\tau )} U_1^p \(  PU_1-U_1\)dx=  \int\limits_{B(\xi_1,\tau )} U_1^p \( -\alpha_n\delta_1^{n-2\over 2}H(x,\xi_1)+R_{\delta_1,\xi_1}\)dx\nonumber \\
 =&  -\alpha_n\delta_1^{n-2\over 2}\int\limits_{B(\xi_1,\tau )} U_1^pH(x,\xi_1)dx+\int\limits_{B(\xi_1,\tau )} U_1^{p }R_{\delta_1,\xi_1} dx,
\end{align*}
with
$$  \int\limits_{B(\xi_1,\tau )} U_1^{p }R_{\delta_1,\xi_1} dx= O\(\({\delta_1\over\tau_1}\)^n\).
$$
By Lemma \ref{lem3} we get
\begin{align*}
&\alpha_n\delta_1^{n-2\over 2}\int\limits_{B(\xi_1,\tau )} U_1^pH(x,\xi_1)dx =\alpha_n^{p+1}\delta_1^{n-2}\int\limits_{B(0,\tau /\delta_1)} H(\delta_1y+\xi_1,\xi_1){1\over (1+|y|^2)^{n+2\over2}}dy\nonumber\\ & =
\alpha_n^{p+1}\({\delta_1\over\tau_1}\)^{n-2} \int\limits_{B(0,\tau /\delta_1)}\tau_1^{n-2}H( \delta_1y+\xi_1,\xi_1){1\over (1+|y|^2)^{n+2\over2}}dy\nonumber\\&=\alpha_n^{p+1}\({\delta_1\over\tau_1}\)^{n-2} \[{1\over 2^{n-2}}\int\limits_{\rr^n)} {1\over(1+|y|^2)^{n+2\over2}}dy+o(1)\].
\end{align*}

{\em Proof of (ii)}
By Lemma \ref{lem2} and Lemma \ref{lem3} we get
\begin{align*}
 &\int\limits_{B(\xi_1,\tau )}    U_1^p  PU_2dx=\int\limits_{B(\xi_1,\tau )}    U_1^p \(U_2-\alpha_n\delta_2^{n-2\over2}H(x,\xi_2)+R_{\delta_2,\xi_2}\)dx\nonumber \\&
 =\alpha_n^{p+1} (\delta_1\delta_2)^{n-2\over2}\int\limits_{B(0,\tau /\delta_1)}{1\over(1+|y|^2)^{n+2\over2}}
 {1\over(\delta_2^2+|\delta_1y+\xi_1-\xi_2|^2)^{n-2\over2}}dy\nonumber\\
 & -\alpha_n^{p+1} (\delta_1\delta_2)^{n-2\over2}\int\limits_{B(0,\tau /\delta_1)}{1\over(1+|y|^2)^{n+2\over2}} H(\delta_1y+\xi_1,\xi_2) dy\nonumber \\ &+\alpha_n^{p+1} (\delta_1\delta_2)^{n-2\over2}\int\limits_{B(0,\tau /\delta_1)}{1\over(1+|y|^2)^{n+2\over2}}
 R_{\delta_2,\xi_2}(\delta_1y+\xi_1)dy=\nonumber \\&
 =\alpha_n^{p+1} {(\delta_1\delta_2)^{n-2\over2}\over |\tau_1-\tau_2|^{n-2}}\int\limits_{B(0,\tau /\delta_1)}{ 1\over(1+|y|^2)^{n+2\over2}}
 {|\tau_1-\tau_2|^{n-2}\over(\delta_2^2+|\delta_1y+\xi_1-\xi_2|^2)^{n-2\over2}}dy\nonumber\\
 & -\alpha_n^{p+1} {(\delta_1\delta_2)^{n-2\over2}\over  |\tau_1+\tau_2|^{n-2}}\int\limits_{B(0,\tau /\delta_1)}{|\tau_1+\tau_2|^{n-2}\over(1+|y|^2)^{n+2\over2}} H(\delta_1y+\xi_1,\xi_2) dy\nonumber \\ &+O\( (\delta_1\delta_2)^{n-2\over2}{\delta_2^{n+2\over2}\over\tau_2^{n}}\)=\nonumber \\&
 =\alpha_n^{p+1} {(\delta_1\delta_2)^{n-2\over2}\over |\tau_1-\tau_2|^{n-2}}\[\int\limits_{\rr^n}{ 1\over(1+|y|^2)^{n+2\over2}}dy+o(1)\]
  \nonumber\\
 & -\alpha_n^{p+1} {(\delta_1\delta_2)^{n-2\over2}\over  |\tau_1+\tau_2|^{n-2}}\[\int\limits_{\rr^n}{1\over(1+|y|^2)^{n+2\over2}} dy +o(1)\]\nonumber \\ &+o\( {(\delta_1\delta_2)^{n-2\over2} \over\tau_2^{n-2}}\).
   \end{align*}

{\em Proof of (iii) and (iv)}
We argue as in the proof of (i) and (ii) using estimates \eqref{lew15} and \eqref{lew16}.

{\em Proof of (v)}
We have
\begin{align}\label{lew11}
\int\limits_\Omega{1\over |x|}U_1 ^{p+1}dx=\int\limits_{B(\xi_1,\tau)}U_1 ^{p+1}dx+ \int\limits_{\Omega\setminus B(\xi_1,\tau)}U_1 ^{p+1}dx,\end{align}
with
$$
   \int\limits_{\Omega\setminus B(\xi_1,\tau)} {1\over |x|}U_1 ^{p+1}dx= O\({\delta_1^n\over\tau^n}\),$$
So, we only have to estimate the first term in \eqref{lew11}.
We split it as
$$  \int\limits_{B(\xi_1,\tau)}  {1\over |x|} U_1 ^{p+1}dx= \int\limits_{B(\xi_1,\tau)} U_1 ^{p+1}dx+\int\limits_{B(\xi_1,\tau)}\({1\over |x|}-1\)U_1 ^{p+1}dx.$$
 We have
$$
\int\limits_{B(\xi_1,\tau)} U_1 ^{p+1}dx=\gamma_1+O\({\delta_1^n\over\tau^n}\).$$
Since $\xi_1=\xi_0(1+\tau_1)$ and $|\xi_0|=1$, by the mean value theorem we get
\begin{align}\label{lew15}
{1\over|\delta_1 y+\tau_1\xi_0+\xi_0|}-1=-\tau_1-\delta_1 \<y, \xi_0\>+R(y),\end{align}
where $R$ satisfies the uniform estimate
\begin{align}\label{lew16}
|R(y )|\le c\(\delta_1^2|y|^2+\delta_1\tau_1|y|+\tau_1^2\)\ \hbox{for any}\ y\in B(0,\tau/\delta_1).\end{align}
Therefore we conclude
\begin{align*}
&   \int\limits_{B(\xi_1,\tau)}\({1\over |x|}-1\)U_1 ^{p+1}dx= \alpha_n^{p+1}\int\limits_{B(0,\tau/\delta_1)}\({1\over |\delta_1 y+\tau_1\xi_0+\xi_0|}-1\){1\over(1+|y|^2)^{n}}dy\nonumber\\ &=\alpha_n^{p+1}\int\limits_{B(0,\tau/\delta_1)}\(-\tau_1-\delta_1\tau_1\<y, \xi_0\>+R(y)\){1\over(1+|y|^2)^{n}}dy
=-\gamma_1\tau_1+o(\tau).\end{align*}
Collecting all the previous estimates we get the claim.
\end{proof}

\begin{lemma}\label{lem3} Let $\tau$ as in \eqref{tau}.
It holds true that
\begin{itemize}\item[(i)]$$\int\limits_{B(0,\tau /\delta_1)}\tau_1^{n-2}H( \delta_1y+\xi_1,\xi_1){1\over (1+|y|^2)^{n+2\over2}}dy={1\over 2^{n-2}}\int\limits_{\rr^n} {1\over(1+|y|^2)^{n+2\over2}}dy+o(1) ,$$
\item[(ii)]$$\int\limits_{B(0,\tau /\delta_1)}{|\tau_1+\tau_2|^{n-2}\over(1+|y|^2)^{n+2\over2}} H(\delta_1y+\xi_1,\xi_2) dy=  \int\limits_{\rr^n} {1\over(1+|y|^2)^{n+2\over2}}dy+o(1) ,$$
    \item[(iii)]$$\int\limits_{B(0,\tau /\delta_1)}{ 1\over(1+|y|^2)^{n+2\over2}}
 {|\tau_1-\tau_2|^{n-2}\over(\delta_2^2+|\delta_1y+\xi_1-\xi_2|^2)^{n-2\over2}}dy=\int\limits_{\rr^n} {1\over(1+|y|^2)^{n+2\over2}}dy+o(1)  .$$
        \end{itemize}
\end{lemma}
\begin{proof}
We are going to use Lebesgue's dominated convergence Theorem together with Lemma \ref{lemacca} .
First of all,   taking into account that $ \xi_1=(1+\tau_1)\xi_0$ and $ \bar \xi_1=(1-\tau_1)\xi_0$ we deduce that
$$\tau_1^{n-2}H( \delta_1y+\xi_1,\xi_1){1\over (1+|y|^2)^{n+2\over2}}\to {1\over 2^{n-2}}{1\over (1+|y|^2)^{n+2\over2}}\ \hbox{a.e. in}\ \rr^n $$
and also that
$$ H( \delta_1y+\xi_1,\xi_1) \le C_2 {1\over|\delta_1y+\xi_1-\bar \xi_1|^{n-2}}=C_2{1\over|\delta_1y+2\tau_1\xi_0|^{n-2}}\le C_2{1\over\tau_1^{n-2}},$$
since
$$|\delta_1y+2\tau\xi_0|\ge 2\tau_1-|\delta_1y|\ge \tau_1\ \hbox{for any} \ y\in B(0,\tau /\delta_1).$$
 That proves (i).

 In a similar way, taking into account that $ \xi_1=(1+\tau_1)\xi_0$ and $ \bar \xi_2=(1-\tau_2)\xi_0$
 we get
 $$(\tau_2+\tau_1)^{n-2}H( \delta_1y+\xi_1,\xi_2){1\over (1+|y|^2)^{n+2\over2}}\to  {1\over (1+|y|^2)^{n+2\over2}}\ \hbox{a.e. in}\ \rr^n $$
and also that
$$ H( \delta_1y+\xi_1,\xi_2) \le C_2 {1\over|\delta_1y+\xi_1-\bar \xi_2|^{n-2}}=C_2{1\over|\delta_1y+(\tau_1+\tau_2)\xi_0|^{n-2}}\le C_2{1\over\tau_2^{n-2}},$$
since
$$|\delta_1y+(\tau_1+\tau_2)\xi_0|\ge   \tau_1+\tau_2 -|\delta_1y|\ge \tau_2\ \hbox{for any} \ y\in B(0,\tau /\delta_1).$$
 That proves (ii).

Finally, we have
$${ 1\over(1+|y|^2)^{n+2\over2}}
 {|\tau_1-\tau_2|^{n-2}\over(\delta_2^2+|\delta_1y+\xi_1-\xi_2|^2)^{n-2\over2}}\to  {1\over (1+|y|^2)^{n+2\over2}}\ \hbox{a.e. in}\ \rr^n $$
and also that
$$
 {1\over(\delta_2^2+|\delta_1y+\xi_1-\xi_2|^2)^{n-2\over2}}\le {1\over  |\delta_1y+\xi_1-\xi_2| ^{n-2 }}\le {2^{n-2}\over|\tau_1-\tau_2|^{n-2}}$$
since
$$|\delta_1y+\xi_1-\xi_2| \ge |\xi_1-\xi_2| -|\delta_1y|\ge {|\xi_1-\xi_2|\over2}  \ \hbox{for any} \ y\in B(0,\tau /\delta_1).$$
 That proves (iii).

 \end{proof}

\appendix
\section{}
 Here we recall some well known facts which are useful to get estimates in Section \ref{due}.

We denote by $G(x,y)$ the Green's function associated to $-\Delta$ with Dirichlet boundary condition and $H(x,y)$ its regular part, i.e.
$$-\Delta_x G(x,y)  = \delta_y (x)  \quad \text{for } x \in \Omega,\quad
G(x,y)   = 0   \quad \text{for } x \in \partial\Omega,
$$
and
$$G(x,y) = \gamma_n \( \frac{1}{|x-y|^{n-2}} - H(x,y) \) \quad \text{ where} \quad \gamma_n = \frac{1}{(n-2)|S^{n-1}|}$$
($|S^{n-1}| = (2 \pi^{n/2})/\ \Gamma(n/2)$ denotes the Lebesgue measure of the $(n-1)$-dimensional unit sphere).

The following   lemma was proved in \cite{R}.

\begin{lemma}\label{lem2}
It holds true that
$$PU_{\de,\xi}(x)=U_{\de,\xi}(x)-\alpha_n\delta^{n-2\over2} H(x,\xi)+O\({\de^{n+2\over2}\over {\rm dist}(\xi,\partial\Omega)^n}\)$$
for any $x\in\Omega.$\end{lemma}
\medskip

Since $\Omega$ is smooth, we can choose small $\epsilon > 0$ such that, for every $x \in \Omega$ with
${\rm dist}(x, \partial\Omega) \le \epsilon$, there is a unique point $x_{\nu} \in \partial\Omega$ satisfying ${\rm dist}(x, \partial\Omega)  = |x - x_{\nu}|$.
For such $x \in \Omega$, we define $x^* = 2x_{\nu} - x$ the reflection point of $x$ with respect to $ \partial\Omega  $.

\medskip
The following two lemmas are proved in \cite{ACP}.
\begin{lemma}\label{lemacca}
It holds true that
$$\left| H(x,y)-{1\over |\bar x-y|^{n-2}}\right|=O\({{\rm dist}(x,\partial\Omega)^n \over  |\bar x-y|^{n-2}}\)$$
and
$$\left|\nabla_x\( H(x,y)-{1\over |\bar x-y|^{n-2}}\)\right|=O\({1 \over  |\bar x-y|^{n-2}}\)$$
for any $x \in \Omega$ with
${\rm dist}(x, \partial\Omega) \le \epsilon$.\end{lemma}

\end{document}